\documentclass[12pt]{amsart}

\newtheorem{theorem}{Theorem}[section]

\newtheorem{corollary}[theorem]{Corollary}

\newtheorem{lemma}[theorem]{Lemma}

\theoremstyle{definition}

\newtheorem{example}[theorem]{Example}

\theoremstyle{parrafo}

\begin{document}

\title[]{Sharp bounds for the difference between the arithmetic and
geometric means}

\author{J. M. Aldaz}
\address{Departamento de Matem\'aticas,
Universidad  Aut\'onoma de Madrid, Cantoblanco 28049, Madrid, Spain.}
\email{jesus.munarriz@uam.es}


\thanks{The author was partially supported by Grant MTM2009-12740-C03-03 of the
D.G.I. of Spain}







\begin{abstract} We present sharp  bounds for $\sum_{i=1}^n \alpha_i x_i - \prod_{i=1}^n x_i^{\alpha_i}$ in terms of the variance of the vector $(x_1^{1/2},\dots,x_n^{1/2})$.
\end{abstract}


\maketitle


\markboth{J. M. Aldaz}{AM-GM}

\section{Introduction} Let us start by fixing some notation.
We use $X$ to denote the vector with non-negative
entries $(x_1,\dots,x_n)$. Of
course, $X$ can also be regarded as the function $X:\{1,\dots, n\}\to [0,\infty)^n$
satisfying $X(i) = x_i$. Then for $g : [0,\infty)\to [0,\infty)$, $g(X)$ is defined as the usual composition of functions. In particular, if $g(t) = t^{1/2}$, $X^{1/2} = (x_1^{1/2},\dots,x_n^{1/2})$.
Given a sequence of weights $\alpha = (\alpha_1,\dots,\alpha_n)$ with $\alpha_i > 0$ and $\sum_{i=1}^n \alpha_i = 1$, and a vector $Y = (y_1,\dots,y_n)$, the variance of $Y$ with respect to $\alpha$
is $\operatorname{Var}_\alpha (Y) = \sum_{i=1}^n \alpha_i \left(y_i - \sum_{k=1}^n \alpha_k y_k \right)^2
= \sum_{i=1}^n \alpha_i y_i^2 - \left(\sum_{k=1}^n \alpha_k y_k \right)^2.$ When $\alpha = (1/n, \dots, 1/n)$ we simply write $\operatorname{Var}(Y)$.
We also use $E_\alpha(Y) := \sum_{k=1}^n \alpha_k y_k$ and
$\Pi_\alpha Y := \prod_{i=1}^n y_i^{\alpha_i}$, with
$E(Y)$ and
$\Pi Y$ denoting the equal weights case. Finally, $Y_{\max}$ and
$Y_{\min}$ respectively stand for the maximum and the minimum values of $Y$.

The inequality between
arithmetic and geometric means
$0 \le E_\alpha X - \Pi_\alpha X$ is self-improving in several ways.
In particular,
it immediately entails that $\operatorname{Var}(X^{1/2})\le E_\alpha X - \Pi_\alpha X $: Just write $E_\alpha X -
(E_\alpha X^{1/2})^2 \le E_\alpha X - (\Pi_\alpha X^{1/2})^2$
(this already has useful consequences, as observed in
 \cite{A1}). Conceptually, variance bounds for $E_\alpha X - \Pi_\alpha X$
represent the natural extension of the equality case in the AM-GM
inequality (zero variance is equivalent to equality). Here we prove that
\begin{equation*}
\frac{1}{1-\alpha_{\min}}\operatorname{Var}_\alpha(X^{1/2})\le E_\alpha X - \Pi_\alpha X
\le
\frac{1}{\alpha_{\min}} \operatorname{Var_\alpha}(X^{1/2}),
\end{equation*}
and both bounds are sharp. We also present a standard application to
H\"older's inequality. The author is indebted to Prof. A. Bravo for
some helpful comments.

\section{Sharp bounds and applications}

Since we seek bounds in terms of variances (but independent of the
specific entries of $X$ itself)
and since $E_\alpha X - \Pi_\alpha X$ is 1-homogeneous (so for $t > 0$,
$E_\alpha tX - \Pi_\alpha tX = t\left(E_\alpha X - \Pi_\alpha X\right)$)
the corresponding bounds must also be 1-homogeneous. This  restricts our choices
 to essentially two possibilities: Either find bounds in terms of the
standard deviation $\sigma(X)$ of $X$, or in terms of the variance of $X^{1/2}$.
However $\sigma(X)$ does not satisfy any lower bound, as the following example shows, so we are left with just one possibility.

\begin{example} For no constant $c >0$ does the inequality
$c \sigma (X) \le  E_\alpha X - \Pi_\alpha X$ always hold. To see
this, just take $n=2$, $\alpha =(1/2, 1/2)$, and
$X = (1+ \varepsilon, 1- \varepsilon)$. Then $\sigma(X) = \varepsilon$,
while $E X - \Pi X = O(\varepsilon^2)$, so the assertion follows by letting $\varepsilon\downarrow 0$.
\end{example}

\begin{theorem}\label{AMGMVar1}  For $n\ge 2$ and $i=1,\dots, n$, let $X = (x_1,\dots,x_n)$ be such that $x_i\ge 0$, and
let $\alpha = (\alpha_1,\dots,\alpha_n)$ satisfy
$\alpha_i > 0$ and $\sum_{i=1}^n \alpha_i = 1$.  Then
\begin{equation}\label{AMGMVar1eq}
\frac{1}{1-\alpha_{\min}}\operatorname{Var}_\alpha(X^{1/2})\le E_\alpha X - \Pi_\alpha X
\le
\frac{1}{\alpha_{\min}} \operatorname{Var_\alpha}(X^{1/2}).
\end{equation}
In particular, if $\alpha = (n^{-1},\dots,n^{-1})$, then
\begin{equation}\label{AMGMVar1eqeq}
\frac{n}{n-1}\operatorname{Var}(X^{1/2})\le E X - \Pi X
\le
n \operatorname{Var}(X^{1/2}).
\end{equation}
\end{theorem}

\begin{example} Note that when $n=2$ the left and right hand sides of (\ref{AMGMVar1eqeq})
are equal, so in general  neither bound can be improved. In fact,
equality can be attained on both sides of (\ref{AMGMVar1eqeq})
 for arbitrary
values of $n$. To see this, on the left hand side let $n >1$
 and let
$x_1 = 1$, $x_2 = \cdots  x_n = 0$. Since
 $\alpha = (n^{-1},\dots,n^{-1})$, we have $\frac{1}{\alpha_{\min}} =n$,
$EX= \frac{1}{n}$,  and $\operatorname{Var}(X^{1/2})= \frac{n-1}{n^2}$, so equality holds.
For the right hand side,  let
$x_1 = \cdots =x_{n-1} = 1$, and $x_n = 0$.
Then $\frac{1}{\alpha_{\min}} =n$, $EX= \frac{n-1}{n}$,  and $\operatorname{Var}(X^{1/2})= \frac{n-1}{n^2}$, so
again equality holds.
\end{example}

The preceding result is motivated by \cite[Theorem]{CaFi}, which states that if $0 < X_{\min}$,
then
\begin{equation}\label{cafieq}
\frac{1}{2 X_{\max} }  \operatorname{Var}_\alpha(X)
\le
E_\alpha X - \Pi {}_\alpha X
\le
\frac{1}{2 X_{\min}}  \operatorname{Var}_\alpha(X)
\end{equation}
 (cf.  also \cite{Alz}, \cite{Me},  \cite{A2}, \cite{A5} for additional
refinements and references, and \cite{A3} for probabilistic information
regarding the GM-AM ratio).

A drawback of (\ref{cafieq}) is that since the inequalities depend on $X_{\max}$ and $X_{\min}$,
they are not well suited for standard arguments where pointwise inequalities are integrated. As an instance, to obtain a refinement of  H\"older's
 Inequality from (\ref{cafieq}), one would have to assume a priori that functions are bounded
 away from 0 and $\infty$, which is  too restrictive, while using (\ref{AMGMVar1eq}) does not
require any such assumption. The standard argument used to derive
 H\"older's inequality from the AM-GM inequality  applies verbatim
to refinements  (cf. \cite[Theorem 2.2]{A4}
for the case of two functions, and \cite[Corollary 2]{A1} for the upper
bound with a weaker constant).

Regarding the meaning of the inequalities
below, they just say that the ``more different" the functions are,
the smaller their product is, and viceversa. Now, since in principle these functions
belong to different spaces, to compare them they are first normalized,
and then mapped to $L^2$ via the Mazur's map (which
has controlled distortion) so differences are measured in $L^2$.

\begin{corollary}\label{betterhold}  For $i = 1,\dots, n$, let $1 < p_i < \infty$
be such that $p_1^{-1} + \cdots + p_n^{-1} = 1$, and let $0\le f_i\in L^{p_i}$
satisfy  $\|f_i\|_{p_i}  > 0$.  Then
\begin{equation}\label{bonhold1}
\prod_{i=1}^n\|f_i\|_{p_i} \left(1 - p_{\max}\sum_{i=1}^n \frac1{p_i}
\left\|\frac{f_i^{p_i/2}}{\|f_i\|_{p_i}^{p_i/2}} -
\sum_{k=1}^n \frac1{p_k} \frac{f_k^{p_k/2}}{\|f_k\|_{p_k}^{p_k/2}}\right\|_2^2\right)_+ \le
\end{equation}
\begin{equation}\label{bonhold2}
\left\|\prod_{i=1}^n f_i\right\|_1
\le
\prod_{i=1}^n\|f_i\|_{p_i} \left(1 - \frac{p_{\max}}{p_{\max} -1}\sum_{i=1}^n \frac1{p_i}
\left\|\frac{f_i^{p_i/2}}{\|f_i\|_{p_i}^{p_i/2}} -
\sum_{k=1}^n \frac1{p_k} \frac{f_k^{p_k/2}}{\|f_k\|_{p_k}^{p_k/2}}\right\|_2^2\right).
\end{equation}
\end{corollary}

\begin{proof}  Set $\alpha_i = p_i^{-1}$ and $x_i = f_i^{p_i}(u)/\|f_i\|_{p_i}^{p_i}$ in (\ref{AMGMVar1eq}). To obtain
(\ref{bonhold1}, \ref{bonhold2}), integrate
and multiply all terms by $\prod_{i=1}^n\|f_i\|_{p_i}$.
\end{proof}

The bounds in Theorem \ref{AMGMVar1} can be used to obtain
new bounds in terms of other variances.

\begin{corollary} For $n\ge 2$ and $i=1,\dots, n$, let $X = (x_1,\dots,x_n)$ be such that $x_i\ge 0$, and
let $\alpha = (\alpha_1,\dots,\alpha_n)$, $\beta = (\beta_1,\dots,\beta_n)$
 satisfy $\alpha_i, \beta_i > 0$ and $\sum_{i=1}^n \alpha_i = \sum_{i=1}^n \beta_i = 1$.  Then
\begin{equation}\label{diffVar}
\min_{k=1, \dots, n}\left\{\frac{\alpha_{k}}{\beta_{k}}\right\}\
\max\left\{\frac{1}{1-\alpha_{\min}}, \frac{1}{1-\beta_{\min}}\right\}\operatorname{Var}_\beta(X^{1/2})\le
\end{equation}
\begin{equation}\label{diffVar1}
E_\alpha X - \Pi_\alpha X
\le
\max_{k=1, \dots, n}\left\{\frac{\alpha_{k}}{\beta_{k}}\right\}\
 \min\left\{\frac{1}{\alpha_{\min}}, \frac{1}{\beta_{\min}}\right\}\operatorname{Var_\beta}(X^{1/2}).
\end{equation}
\end{corollary}

\begin{proof} From Theorem \ref{AMGMVar1}
and \cite[Theorem 2.1]{A1} , which states that
\begin{equation}\label{refAMGMFugen}
\min_{k=1, \dots, n}\left\{\frac{\alpha_{k}}{\beta_{k}}\right\}\left(E_\beta X - \Pi _\beta X\right)
\le
E_\alpha X - \Pi_\alpha X
\le
\max_{k=1, \dots, n}\left\{\frac{\alpha_{k}}{\beta_{k}}\right\}\left(E_\beta X - \Pi_\beta X\right),
\end{equation}
we immediately obtain
\begin{equation*}
\min_{k=1, \dots, n}\left\{\frac{\alpha_{k}}{\beta_{k}}\right\}\
\frac{1}{1-\beta_{\min}}\operatorname{Var}_\beta(X^{1/2})\le
E_\alpha X - \Pi_\alpha X
\le
\max_{k=1, \dots, n}\left\{\frac{\alpha_{k}}{\beta_{k}}\right\}\
 \frac{1}{\beta_{\min}}\operatorname{Var_\beta}(X^{1/2}).
\end{equation*}
The analogous bounds in terms of $\alpha_{\min}$ follow from the fact that
 given  any vector $Y$, not necessarily positive,
\begin{equation}
\min_{k=1, \dots, n}\left\{\frac{\alpha_{k}}{\beta_{k}}\right\}\
 \operatorname{Var}_\beta(Y)\le \operatorname{Var}_\alpha(Y)
\le
\max_{k=1, \dots, n}\left\{\frac{\alpha_{k}}{\beta_{k}}\right\} \operatorname{Var_\beta}(Y),
\end{equation}
which is a special case of the Dragomir-Jensen inequality
(cf. \cite{Dra} for the original inequality, proven in the discrete
case, and \cite{A6} for a general version).
\end{proof}

\section{Proof of the Theorem}

As  in \cite{CaFi}, we use an induction argument, so the first step is
to prove the inequality when $n=2$ (Lemmas \ref{equal2} and \ref{ineq2}). Unlike
\cite{CaFi}, since no a priori bounds are imposed on $X$, we need to
take into account  the possibility that one or several entries of $X$
be zero (Lemma \ref{zero}).

\begin{lemma}\label{equal2}  For all $x,y \ge 0$, setting $X = (x,y)$
we have
\begin{equation}\label{equal2eq}
\frac{x + y}{2} - \sqrt{ x y} = 2 \operatorname{Var}(X^{1/2}).
\end{equation}
\end{lemma}

\begin{proof} Expand the right hand side.
\end{proof}

Next  we use the following strengthening of Young's inequality,
which appeared in \cite[Lemma 2.1]{A4} under a slightly different
notation (cf.  \cite{Fu} and \cite{A3} for generalizations).

\begin{lemma}\label{betteryoungl}  Let $a \in [0, 1/2]$.  Then for all $x,y \ge 0$
\begin{equation}\label{betteryoung}
2 a \left(\frac{x + y}{2}  - \sqrt{ x y} \right)  \le
a x  + (1-a) y - x^a y^{1-a} \le 2 (1-a) \left( \frac{x + y}{2}
 - \sqrt{ x y} \right) .
\end{equation}
\end{lemma}

To rewrite \cite[Lemma 2.1]{A4} as above, make the change of variables
$a=1/q$, $1-a = 1/p$, $x^a = v$, $y^{1-a} = u$, and expand the squares.

\begin{lemma}\label{ineq2}  Let $a \in (0, 1/2]$, and set
$\alpha = (a, 1-a)$. Writing $X = (x,y)$, where  $x,y \ge 0$, we have
\begin{equation}\label{ineq2eq}
\frac{1}{1 - a} \operatorname{Var}_\alpha(X^{1/2})  \le
a x  + (1-a) y - x^a y^{1-a} \le \frac{1}{a} \operatorname{Var}_\alpha(X^{1/2}).
\end{equation}
\end{lemma}

\begin{proof} Consider first the right hand side. Write
$$
f(x,y) := \frac{1}{a} \operatorname{Var}_\alpha(X^{1/2})
- a x  - (1-a) y + x^a y^{1-a}.
$$
 To see that $f(x,y) \ge 0$,
simplify first. This yields
$$
f(x,y) := (1-2a) x -
2(1-a)\sqrt{ x y}  + x^a y^{1-a}.
$$
Since
$\sqrt{ x y} = \frac{x + y}{2} -  2 \operatorname{Var}(X^{1/2})$ by
(\ref{equal2eq}), substituting and simplifying we find that $f(x,y) \ge 0$ if and only if
$$a x  + (1-a) y - x^a y^{1-a} \le
2(1-a) 2 \operatorname{Var}(X^{1/2}),$$
 which is just the second inequality
in (\ref{betteryoung}),
together with  (\ref{equal2eq}).

For the left hand side inequality in (\ref{ineq2}), follow essentially
the same steps as before, but using the first inequality
in (\ref{betteryoung}) instead of the second.
\end{proof}

\begin{lemma}\label{zero}  For $n\ge 2$ and $i=1,\dots, n$, let $X = (x_1,\dots,x_n)$ be such that $x_i\ge 0$, and  $x_j = 0$
for some index $j$.
Let
$\alpha_i > 0$ satisfy $\sum_{i=1}^n \alpha_i = 1$.  Then
\begin{equation}\label{zeroeq}
\frac{1}{1-\alpha_{\min}} \operatorname{Var_\alpha}(X^{1/2})
\le  E_\alpha X
\le
\frac{1}{\alpha_{\min}} \operatorname{Var_\alpha}(X^{1/2}).
\end{equation}
\end{lemma}

\begin{proof} We prove the right hand side inequality. The argument
for the left hand side is entirely analogous. Define, for non-negative  $X = (x_1,\dots,x_n)$,
$$
f(X) := \frac{1}{\alpha_{\min}} \operatorname{Var}_\alpha(X^{1/2}) - E_\alpha X
= \frac{1}{\alpha_{\min}}\left(\sum_{i=1}^n \alpha_i x_i - \left(\sum_{i=1}^n \alpha_i x_i^{1/2}\right)^2\right) - \sum_{i=1}^n \alpha_i x_i.
$$
 To see that if some coordinate equals zero then $f(X) \ge 0$, we use induction on the number of
 non-zero coordinates. Suppose first that exactly one
coordinate in $X$ is different from  zero, say
 $x_i >0$. Then
$$
f(X) = \frac{1}{\alpha_{\min}}\left(\alpha_i x_i - \left(\alpha_i x_i^{1/2}\right)^2\right) - \alpha_i x_i =
\left( \frac{1 -\alpha_{\min} - \alpha_i}{\alpha_{\min}}\right) \alpha_i x_i
\ge 0,
$$
with equality if and only if $n=2$ and $\alpha_i \ne \alpha_{\min}$.

Next, suppose  that exactly $k$
coordinates in $X = (x_1,\dots,x_n)$ are larger that  zero,
for $2\le k\le n-1$, and that the result is true whenever fewer
than $k$ coordinates are non-zero. Without loss  of
generality, we may assume that $x_1,\dots, x_k >0$. Since
$f(X)$ is $1$-homogeneous, that is, for every $t>0$, $f(tX) = t f(X)$,
we may also assume that $
\sum_{i=0}^k \alpha_i x_i =1.
$
Write $h(x_1, \dots, x_k)= f(x_1, \dots,x_k, 0, \dots, 0)$. By
compactness of the simplex $S:=\{(x_1, \dots, x_k) \in\mathbb{R}^k:
x_i\ge 0$ for $i = 1,\dots, k,$ and $\sum_{i=0}^k \alpha_i x_i =1\}$, $h$ has a global minimum  on $S$. If the minimum  is achieved at
the boundary, then $h\ge 0$ on $S$
by the induction hypothesis, so
it is enough to check that  $h(y_1, \dots, y_k) \ge 0$ whenever
  $(y_1, \dots, y_k) $ is a critical point of $h$ in the relative
  interior of $S$. Using Lagrange
  multipliers, we obtain, for $j = 1,\dots, k,$
$$
h_j(x_1, \dots, x_k) =
\frac{1}{\alpha_{\min}}\left(\alpha_j  - \alpha_j \left(\sum_{i=1}^k \alpha_i x_i^{1/2}\right) x_j^{-1/2} \right) - \alpha_j = \lambda \alpha_j.
$$
Simplifying we find that
$$  - \frac{1}{\alpha_{\min}}\left(\sum_{i=1}^k \alpha_i x_i^{1/2}\right) x_j^{-1/2}  = \lambda + 1 - \frac{1}{\alpha_{\min}},
$$
and since the left hand side is not zero, so is the right hand side.
Thus,
$$
x_j^{1/2} = \frac{\sum_{i=1}^k \alpha_i x_i^{1/2}}{1- \alpha_{\min} - \lambda  \alpha_{\min}},
$$
and it follows that whenever $(y_1, \dots, y_k) $ is a critical
point, all its coordinates are equal, say, to the value $t$
defined by
$$
t^{1/2} = \frac{\sum_{i=1}^k \alpha_i y_i^{1/2}}{1- \alpha_{\min} - \lambda  \alpha_{\min}}.
$$
Then
$$
h(t, \dots, t) = \frac{t \sum_{i=1}^k \alpha_i}{\alpha_{\min}}\left(1  - \sum_{i=1}^k \alpha_i -  \alpha_{\min}\ \right) \ge 0,
$$
with equality if and only if $k=n-1$ and $\alpha_{n} = \alpha_{\min}$.
\end{proof}

\vskip .3 cm

\noindent {\em Proof of the theorem.}
We are now ready to show that for every $X\in [0,\infty)^n$,
$$
f(X) := \frac{1}{\alpha_{\min}} \operatorname{Var}_\alpha(X^{1/2}) - E_\alpha X + \Pi_\alpha X
\ge 0,
$$
which proves the right hand side inequality in (\ref{AMGMVar1eq}); the argument
for the left hand side inequality is entirely analogous.
  Again, by 1-homogeneity it is enough to show
that $f(Y) \ge 0$ for every critical point $Y$ of $f$ in the simplex
$S:=\{(x_1, \dots, x_n) \in\mathbb{R}^n:
x_i\ge 0$ for $i = 1,\dots, n,$ and $\sum_{i=0}^n \alpha_i x_i =1\}$.
Now, on the boundary of $S$, $f(X) \ge 0$ by Lemma \ref{zero}. In the relative
interior of $S$ we use Lagrange multipliers together with induction.
The induction hypothesis states that whenever $\beta$ is a sequence
of fewer than $n$ weights (positive and adding up to 1),
and $W\in [0,\infty)^{n-1}$, we have
\begin{equation}\label{IndStep}
 E_\beta W - \Pi_\beta W
\le
\frac{1}{\beta_{\min}} \operatorname{Var_\beta}(W^{1/2}).
\end{equation}
If $n= 2$ the result holds by Lemma \ref{ineq2}, so suppose that
$n\ge 3$, and let $Y$ be a critical point of $f$ in the relative
interior of $S$. Then
$$
f_j(y_1, \dots, y_n) =
\frac{1}{\alpha_{\min}}\left(\alpha_j  -  \alpha_j  \left(\sum_{i=1}^k \alpha_i y_i^{1/2}\right) y_j^{-1/2} \right) - \alpha_j  +  \alpha_j  \left(\prod_{i=1}^n y_i^{\alpha_i}\right) y_j^{-1}
= \lambda \alpha_j.
$$
Simplifying, this yields
$$
\left(\frac{1}{\alpha_{\min}} - 1 - \lambda\right) y_j - \left(\frac{\sum_{i=1}^k \alpha_i y_i^{1/2}}{\alpha_{\min}}\right) y_j^{1/2}   + \prod_{i=1}^n y_i^{\alpha_i}
= 0.
$$
Writing $A = 1/\alpha_{\min} - 1 - \lambda$,
$B=\sum_{i=1}^n \alpha_i y_i^{1/2}/\alpha_{\min}$,
and $C= \prod_{i=1}^n y_i^{\alpha_i}$
we find
that $y_1^{1/2},\dots,y_n^{1/2}$ are all positive solutions of
$$
A t^2 - B t + C=0.
$$
Given that this equation
 has at most two roots and
 $n\ge 3$,  at least
two coordinates of $Y$ are equal. By relabeling if needed, we may assume
that  $y_{n-1} = y_n$. Set, for $k<n-1$, $\beta_k = \alpha_k$,
$w_k = y_k$, and define
$\beta_{n-1} = \alpha_{n-1} + \alpha_{n}$,   $w_{n-1} = y_{n-1}$.
With
$\beta := (\beta_1,\dots,\beta_{n-1})$ and $W := (w_1,\dots,w_{n-1})$,
 we have
\begin{equation}\label{IndStepap}
 E_\alpha Y - \Pi_\alpha Y
=
 E_\beta W - \Pi_\beta W
\le
\frac{1}{\beta_{\min}} \operatorname{Var_\beta}(W^{1/2})
\le
\frac{1}{\alpha_{\min}} \operatorname{Var_\alpha}(Y^{1/2})
\end{equation}
since $\alpha_{\min} \le \beta_{\min}$. Thus, $f(Y)\ge 0$ at all the critical points $Y\in S$, so
$f(X)\ge 0$ for all non-negative $X$.

For the left hand side inequality in (\ref{AMGMVar1eq}),
note that since  $\alpha_{\min} \le \beta_{\min}$, we have
$(1 -  \alpha_{\min})^{-1} \le (1- \beta_{\min})^{-1}$, and the result follows again by
induction. \qed

\end{document}